\documentclass[11pt]{article}

\usepackage{graphicx}
\usepackage{amsmath,amsfonts,amssymb}   
\usepackage{color} 
\usepackage{comment} 
\usepackage{authblk}
\usepackage{amsthm}
\usepackage{doi}

\newtheorem{theorem}{Theorem}
\newtheorem{lemma}{Lemma}

\begin{document}
\title{Projective Wishart distributions}
\date{}
\author[1]{Emmanuel Chevallier}

\affil[1]{{\small Aix Marseille Univ, CNRS, Centrale Marseille, Institut Fresnel, Marseille 13013, France}}
\maketitle              
\begin{abstract}

We are interested in the distribution of Wishart samples after forgetting their scaling factors. We call such a distribution a projective Wishart distribution. We show that projective Wishart distributions have strong links with the affine-invariant geometry of symmetric positive definite matrices in the real case or Hermitian positive definite matrices in the complex case. First, the Fréchet mean of a projective Wishart distribution is the covariance parameter, up to a scaling factor, of the corresponding Wishart distribution. Second, in the case of $2\times 2$ matrices, the densities have simple expressions in term of the affine-invariant distance.

\end{abstract}

\noindent \textbf{keywords}: Wishart distributions  \and positive definite matrices \and hyperbolic spaces \and Fréchet means.

\section{Introduction}

A Wishart distribution is the law of the empirical second order moment of a set of i.i.d. Gaussian random vectors, see \cite{Wishart} for the original paper of Wishart, or \cite{Muirhead,Kollo} for a more modern presentation. These distributions are parametrized by the covariance $\Sigma$ of the Gaussian and the number $n$ of i.i.d. random vectors. It is well known that when a linear map $A$ acts on the random vectors, a Wishart distribution of parameter $\Sigma$ is turned into a Wishart distribution of parameter $A\Sigma A^{*}$. 

This equivariance property draws a link between Wishart distributions and the affine-invariant geometry of symmetric positive definite and Hermitian positive definite matrices. We are interested here in explicit links between Wishart distributions and affine invariant distance functions.

The link with affine-invariant geometry plays a role in analysis on symmetric cones, see \cite{Terras,Faraut}. In \cite{Terras}, the author shows how the Wishart distribution and its normalizing constant relates to certain integrals on symmetric positive definite matrices. These results rely mostly on the equivariance property itself and the relation between Wishart distributions and the distance function is indirect. 

Authors of \cite{Zhuang} have build an estimator of the parameter $\Sigma$ based on empirical Fréchet means, defined using an affine-invariant distance. In order to obtain a consistent estimator, they need to introduce a multiplicative correcting factor. The underlying reason is that the Euclidean mean $n\Sigma$ and the Fréchet mean of a Wishart distribution do not coincide. 

The link between affine-invariant distances and Wishart distributions appears more clearly when we consider second order moments up to a scaling factor. To do so, we renormalize the second order moments by their determinant and call the corresponding distribution a projective Wishart distribution. These distributions are formally introduced in section \ref{sec:Wishart distributions}. In section \ref{sec:The geometry of S}, we describe the action of invertible linear maps of second order moments and the associated distance function. Section \ref{sec:Properties of projective Wishart distributions} contains the main results. First, the Fréchet mean of a projective Wishart distribution is the covariance parameter, up to a scaling factor, of the corresponding Wishart distribution. Second, in the case of $2\times 2$ matrices, the densities have simple expressions in term of the affine-invariant distance. 

\section{Wishart distributions}
\label{sec:Wishart distributions}
Let $\mathcal{N}_{d}(0,\Sigma)$ denote the Gaussian distribution on column vectors of $K^d$ with $K=\mathbb{R}$ or $K=\mathbb{C}$, whose mean is $0$ and covariance is $\Sigma$. Note $\mathcal{P}_d(K)$ the set of $d\times d$ symmetric positive definite matrices when $K=\mathbb{R}$ and the set of Hermitian positive definite matrices when $K=\mathbb{C}$. In the rest of the paper, $\Sigma \in \mathcal{P}_d(K)$. Let $(Y_i)_{i\in \mathbb{N}}$ be a sequence of i.i.d. random variables with

$$ Y_i \sim  \mathcal{N}_{d}(0,\Sigma).$$

For $n\in \mathbb{N}$, consider the random variable

$$ X_n = \sum_{i=1}^{i=n} Y_i Y_i^T $$ 

The Wishart distribution with parameters $\Sigma$ and $n$, noted $\mathcal{W}(\Sigma,n)$, is defined as the law of the random variable $X_n$. When $n\geq d$, $X_n\in\mathcal P_d(K)$ almost surely and the distribution $\mathcal{W}(\Sigma,n)$ is supported on $\mathcal{P}_d(K)$.

Our aim is to study covariance matrices up to a scaling factor. Hence, for $X\in \mathcal{P}_d(K)$, consider the equivalence classes 

$$ \bar X = \{ \alpha X, \alpha \in \mathbb{R}_{>0} \},$$
and denote $\mathcal{S}=\mathcal{P}_d(K)/\mathbb{R}_{>0}$ the set of equivalence classes. In the rest of the paper $X$ denotes an element of $\mathcal{P}_d(K)$ and $x$ an element of $\mathcal{S}$. There exists several standard parametrizations of $\mathcal{S}$, such as matrices of $\mathcal{P}_d(K)$ of fixed trace or matrices of fixed determinant. In the rest of the paper, we identify $\mathcal{S}$ with matrices of determinant $1$,

$$ \mathcal{S} \sim \{ X \in \mathcal{P}_d(K), \det(X)=1 \}.$$

Note $\pi$ the canonical projection 
$$\pi(X)=\frac{1}{\det(X)^d}X.$$ 

We are interested in the law $P\mathcal{W}(\Sigma,n)$ of $\pi(X)$ when $X$ follows a Wishart distribution $\mathcal{W}(\Sigma,n)$. Distributions $P\mathcal{W}$ will be called projective Wishart distributions.  

In the next sections, we will study properties of such distributions. In particular, we will show that $\pi(\Sigma)$ is the Frechet mean of $P\mathcal{W}(\Sigma,n)$ for the affine-invariant metric on $\mathcal{S}$. We will also show that when $d=2$, the density of $P\mathcal{W}(\Sigma,n)$ evaluated in $x$ has a simple expression depending on the distance between $x$ and $\bar \Sigma$.

\section{The geometry of $\mathcal{S}$}
\label{sec:The geometry of S}

The space $\mathcal{P}_d(K)$ is an open cone of the vector space of (Hermitian-)symmetric matrices. This cone is invariant by the following action of invertible matrices

$$ G \cdot X = GXG^{*}, $$
where $G\in GL_d(K)$, $X\in \mathcal{P}_d(K)$, and where $A^*$ refers to the transpose or the conjugate transpose of $A$.

Let denote $SL_d(K)$ the subset of $GL_d(K)$ of matrices with determinant $1$. The action of $G\in SL_d(K)$ preserves the determinant:

$$ \det(G \cdot X) = \det(X). $$
Hence $SL_d(K)$ preserves $\mathcal{S}$. The so called affine-invariant on $\mathcal{S}$ is defined up to a multiplicative constant by

\begin{equation}
\label{eq:dist}
d(x,y) \propto \| \log(x^{-\frac{1}{2}}yx^{-\frac{1}{2}}) \|,    
\end{equation}

see \cite{Bhatia,Pennec,Thanwerdas}.

This distance $d(.,.)$ will be used to define the Fréchet mean. In the particular case $d=2$, $\mathcal{P}_d(K)$ endowed with the distance $d(.,.)$ becomes a hyperbolic space of dimension $2$ when $K=\mathbb{R}$ or $3$ when $K=\mathbb{C}$. The property of isotropy of hyperbolic spaces enable to express the density of $P\mathcal{W}(\Sigma,n)$ as a function of $d(.,.)$.

\section{Properties of projective Wishart distributions}
\label{sec:Properties of projective Wishart distributions}

In order to prove the results on Fréchet mean and on densities, we need to state two lemmas. 
The first lemma states that the projective Wishart distribution $P\mathcal{W}(\Sigma,n)$ is invariant by a certain subgroup, noted $H_{\Sigma}$, of isometries of $d(.,.)$ which fix $\Sigma$. The second lemma states that $\bar \Sigma$ is the only fixed point of $H_{\Sigma}$ in $\mathcal{S}$. \\

Note $H$ the subset of matrices $R\in SL_d(K)$ with $RR^{*}=I$. When $K=\mathbb{R}$, $H$ is the special orthogonal group $SO_d$ and when $K=\mathbb{C}$, $H$ is the special unitary group $SU_d$. Define $H_{\Sigma}\subset SL_d(K)$ as the set of matrices 

$$R_{\Sigma} = \Sigma^{\frac{1}{2}}R\Sigma^{-\frac{1}{2}}$$

with $R\in H$. Matrices of $H_{\Sigma}$ leave $\Sigma$ stable. Note $R_{\Sigma} \cdot P\mathcal{W}(\Sigma,n)$ the action of $R_{\Sigma}$ on the distribution $P\mathcal{W}(\Sigma,n)$:

$$ R_{\Sigma} \cdot P\mathcal{W}(\Sigma,n)(R_{\Sigma} \cdot A) = P\mathcal{W}(\Sigma,n)(A)$$
for all measurable subset $A$ of $\mathcal{S}$. As $R_{\Sigma}$ leaves $\mathcal{S}$ stable, $R_{\Sigma} \cdot P\mathcal{W}(\Sigma,n)$ is still a distribution on  $\mathcal{S}$. State the first lemma.

\begin{lemma}
\label{lm:invariance}

$\forall R_{\Sigma}\in H_{\Sigma}$,

$$ R_{\Sigma} \cdot P\mathcal{W}(\Sigma,n) = P\mathcal{W}(\Sigma,n).$$

\end{lemma}

\begin{proof}
The invariance for Wishart distributions, namely $ R_{\Sigma} \cdot \mathcal{W}(\Sigma,n) = \mathcal{W}(\Sigma,n), $ is well known and can easily be checked from their definition, see \cite{Muirhead,Kollo}. Projective Wishart distributions are images of Wishart distributions by the projection $\pi$. Since the projection $\pi$ commutes the the action of $Gl_d(K)$, the invariance also holds for projective Wishart distributions.
\end{proof}

\begin{lemma}
\label{lm:fixPoint}

$\bar\Sigma=\pi(\Sigma)$ is the only fixed point of $H_{\Sigma}$ in $\mathcal{S}$.

\end{lemma}

\begin{proof}

Consider first a matrix $X\in \mathcal{P}_d(K)$ fixed by $H$.  Since

$$R X R^* = X \Leftrightarrow R X  = X R,$$
the matrix $X$ preserves the set of fixed points of all matrices $R\in H$. Furthermore, for all $u\in K^d$, there exist  $R_1$ and $R_2\in H$ such that if $R_1(v)=R_2(v)=v$ then $v\in Ku$ (if $K=\mathbb R$ and if $d$ is even, there is no $R\in H$ such that $Ku$ is the set of fixed points). Now $R_i(X(u))=X(R_i(u))=X(u)$, $i=1,2$, hence $X(u)\in Ku$. It follows that $X$ leaves all lines stable, which in turn implies  that $ X = cI$ for some $c\in \mathbb{R}_{>0}$.

If $X$ is fixed by $H_{\Sigma}$ then $\Sigma^{-\frac{1}{2}}\cdot X$ if fixed by $H$. Hence there is $c\in \mathbb{R}_{>0}$ such that $\Sigma^{-\frac{1}{2}}\cdot X = cI$, and $X=c\Sigma$. This proves that $\bar \Sigma$ is the only fixed point in $\mathcal{S}$.
    
\end{proof}

\subsection{Fréchet means}

The Fréchet mean of a probability distribution $\mu$ on $\mathcal{S}$ can be defined as

$$ F(\mu) =  \operatorname{argmin}_x \int_{\mathcal{S}} d(x,y)^2\mathrm{d}\mu(y) $$

\cite{Afsari,Arnaudon}.

\begin{theorem}
Let $\Sigma\in \mathcal{P}_d(K)$. The Fréchet mean of $P\mathcal{W}(\Sigma,n)$ is unique and equals to $\bar \Sigma$ : 
$$ F(P\mathcal{W}(\Sigma,n)) = \frac{1}{\det{\Sigma}^d}\Sigma $$
\end{theorem}

\begin{proof}
Prove first the equivariance of the Fréchet mean with respect to isometries. Let $\varphi$ be an isometry of $d(.,.)$. Let $\mu$ be any probability measure on $\mathcal S$. Note $\varphi_*\mu$ the pushforward of $\mu$, defined as $\varphi_*\mu(A)=\mu(\varphi^{-1}(A))$. Since

$$\int_{\mathcal{S}} d(x,y)^2\mathrm{d}\mu(y)=\int_{\mathcal{S}} d(x,\varphi^{-1}(y))^2\mathrm{d}\varphi_*\mu(y)=\int_{\mathcal{S}} d(\varphi(x),y)^2\mathrm{d}\varphi_*\mu(y),$$
it can be checked that $\varphi(F(\mu))=F(\varphi_*\mu)$. Using lemma \ref{lm:invariance}, the equivariance of the mean, and that $SL_d(K)$ is acting by isometries, we see that 

$$\forall   R_{\Sigma}\in H_{\Sigma},\, R_{\Sigma}\cdot F(P\mathcal{W}(\Sigma,n)) = F(R_{\Sigma} \cdot P\mathcal{W}(\Sigma,n)) = F(P\mathcal{W}(\Sigma,n))$$

from which we deduce with Lemma \ref{lm:fixPoint} that $F(P\mathcal{W}(\Sigma,n)) \subset \{\bar \Sigma \}$: if the mean exists, it is $\bar \Sigma$. In order to show that a Fréchet mean exists, we need to show that 

$$ a(x)=\int_{\mathcal{S}} d(x,y)^2\mathrm{d}\nu(y) $$

is finite for at least one $x\in \mathcal{S}$. Consider the distribution $P\mathcal{W}(I,n)$ and show that $a(I)$ is finite. An equivariance argument can then transfer the reasoning to $P\mathcal{W}(\Sigma,n)$ and $a(\bar \Sigma)$. Note $\lambda_i\in \mathbb{R}_{+}$ the eigenvalues of $X\in \mathcal{P}_d(K)$. A calculation shows that using the distance of Eq (\ref{eq:dist}),

$$d^2(\bar X,I)\propto \sum_i \log^2\left(\frac{\lambda_i}{\prod_j \lambda_j^{\frac{1}{d}}}\right)= \left(\sum_i \log^2(\lambda_i)-\frac{1}{d}\log^2\left(\prod_i\lambda_i\right)\right).$$

From the marginal distribution of the $\lambda_i$ when $X\sim \mathcal{W}(I,n)$, see \cite{Kollo,James}, we deduce that

$$a(I) \propto \int_{\mathbb{R}_+^d}  \left(\sum_i \log^2(\lambda_i)-\log^2\left(\prod_i\lambda_i\right)\right)\left(\prod_i \lambda_i^{k_1}\right) e^{-\frac{1}{2}\sum_i \lambda_i} \prod_{i>j} |\lambda_i - \lambda_j|^{k_2}\mathrm{d}\lambda,$$

where $k_1,k_2$ are positive constants depending on the case $K=\mathbb{R}$ or $K=\mathbb{C}$. Since the $\lambda_i$ are positive, the exponential terms dominate the other terms and the integral converges.

\end{proof}

\subsection{Densities in the $2\times 2$ case}

Let us start  by defining densities on $\mathcal{S}$. It can be proved that up to a multiplicative factor, there exists a unique volume measure $\mathcal{S}$ invariant by the action of $SL_d(K)$. Note $\nu$ such a measure. The density of the probability distribution $P\mathcal{W}(\Sigma,n)$ is defined as a function $f_{P\mathcal{W}}(.;\Sigma,n):\mathcal{S}\rightarrow \mathbb{R}$ such that for any measurable subset $A$ of $\mathcal{S}$

$$ P\mathcal{W}(\Sigma,n)(A) = \int_{A} f_{P\mathcal{W}}(x;\Sigma,n)d\nu(x).$$

To simplify notations, we will only write the parameters $\Sigma$ and $n$ when it is necessary. Since both $P\mathcal{W}(\Sigma,n)$ and $\nu$ are invariant by the action of $H_{\Sigma}$, it can be checked that the density $f$ is also invariant:

$$\text{ for }\nu \text{-almost all } x\in\mathcal S, f_{P\mathcal{W}}(R_{\Sigma}\cdot x)=f(x).$$     

It is easy to prove that the density $f_{P\mathcal{W}}$ can be chosen continuous. In that case, the equality holds for all $x$.

Hence the density $f_{P\mathcal{W}}$ is constant on the orbits of the action of $H_{\Sigma}$. Since $H_{\Sigma}$ is acting by isometries which fix $\bar \Sigma$, the orbits are contained in balls of center $\bar \Sigma$ of fixed radius. When $d=2$, orbits of $H_{\Sigma}$ are the full balls. Indeed, as mentioned in section \ref{sec:The geometry of S}, in that case, $\mathcal{S}$ is an hyperbolic space of dimension $2$ when $K=\mathbb{R}$ and an hyperbolic space of dimension $3$ when $K=\mathbb{C}$. In these cases,
 $\mathcal{S}$ is not only a symmetric space but also an isotorpic space. We have the additional property that 

$$ d(x,\bar\Sigma)=d(y,\bar\Sigma) \implies \exists R_{\Sigma}\in H_{\Sigma}, R_{\Sigma} \cdot x = y .$$
Hence, we have the following theorem.

\begin{theorem}
\label{th:density}
The density $f_{P\mathcal{W}}(x ;\Sigma,n)$ can be factored through a function $h_{\Sigma,n}:\mathbb{R}_+\rightarrow \mathbb{R}_+$,

$$f_{P\mathcal{W}}(x;\Sigma,n)=h_{\Sigma,n}(d(x,\bar\Sigma)), $$ 

where $\bar\Sigma=\frac{1}{\det(\Sigma)^d}\Sigma$.
\end{theorem}

As we will see later in the explicit calculation of $f_{P\mathcal{W}}$, $h_{\Sigma,n}$ does not depend on $\Sigma$. This follows from the transitivity of the action of $GL_d(K)$ on $\mathcal{P}_d(K)$ and the commutation relations
$G\cdot \mathcal{W}(\Sigma,n)=\mathcal{W}(G\cdot \Sigma,n)$ and $G \cdot \pi(X) = \pi(G\cdot X)$.

To compute the density $f_{P\mathcal{W}}$ of $P\mathcal{W}(\Sigma,n)$, we will integrate the density of $\mathcal{W}(\Sigma,n)$ along fibers $\pi^{-1}(x)$ of the projection $\pi$. 
Although  we introduced the measure $\nu$ on $\mathcal{S}$, we shall need a reference measure on the entire set $\mathcal{P}_d(K)$ to define the density of $\mathcal{W}(\Sigma,n)$.
Note first that the determinant of a matrix in $X\in \mathcal{P}_d(K)$ is always a positive real number since the eigenvalues of $X$ are real and positive. Consider the following identification between $\mathcal{P}_d(K)$ and the cartesian product $\mathcal{S} \times \mathbb{R}$,

\begin{center}
\begin{tabular}{cccl}
   $\theta:$ &$\mathcal{P}_d(K)$& $\rightarrow$  & $\mathcal{S} \times \mathbb{R}$ \\
   &  $X$&$\mapsto$&$ \left(\frac{1}{\det(X)^d}X,\log(\det(X))\right)$ ,\\
\end{tabular}
\end{center}
and define the measure $\nu_{tot}$ on $\mathcal{P}_d(K)$ as the product measure between $\nu$ and the Lebesgue measure on $\mathbb{R}$. It can be checked that $\nu_{tot}$ is invariant by the action of $GL_d(K)$. 

Let $f_W$ denote the density of the Wishart distribution with parameters $\Sigma$ and $n$, with respect to $\nu_{tot}$. From the definition of the projective Wishart distribution and the definition of the reference measure $\nu_{tot}$, we have

$$ \int_{A\subset \mathcal{S}} f_{P\mathcal{W}} \mathrm{d}\nu = \int_{\pi^{-1}(A)} f_W \mathrm{d}\nu_{tot} =\int_{x\in A} \int_{\alpha \in \mathbb{R}} f_W(e^{\frac{\alpha}{d}} x)\mathrm{d}\alpha\mathrm{d}\nu,$$

where $\mathrm{d}\alpha$ refers to the Lebesgue measure on $\mathbb{R}$. The factor $\frac{1}{d}$ arise from the $d$-linearity of the determinant. Hence, with the change of variable $\beta=e^{\tfrac{\alpha}{d}}$, 

$$ f_{P\mathcal{W}}(x) = \int_{\mathbb R} f_W(e^{\frac{\alpha}{d}} x)\mathrm{d}\alpha= \int_{\mathbb R_{>0}} f_W(\beta x)\frac{d}{\beta}\mathrm{d}\beta.$$
Now, the Wishart density with respect to $\nu_{tot}$ is given up to a multiplicative constant by

$$ f_{\mathcal{W}}(X) \propto (\det{X})^{\frac{kn}{2}}e^{-\frac{1}{2}\operatorname{tr(\Sigma^{-1}X)}},$$
where $k=1$ when $K=\mathbb{R}$ and $k=2$ when $K=\mathbb{C}$. The calculation in the real case can be found in \cite{Terras}, and the complex case is obtained by the same reasoning. Hence

$$f_{P\mathcal{W}}(x) \propto \int \left(\beta^d\right)^{\frac{kn}{2}} e^{-\frac{\beta}{2}\operatorname{tr(\Sigma^{-1}x)}}\frac{d}{\beta}\mathrm{d}\beta.$$

Set $\gamma = \frac{\beta}{2}\operatorname{tr(\Sigma^{-1}x)}$. The integral becomes

\begin{equation}
\label{eq:denDimD}
 f_{P\mathcal{W}}(x) 
\propto 
\left(\frac{2}{\operatorname{tr(\Sigma^{-1}x)}} \right)^{\frac{dkn}{2}} \int \gamma^{\frac{dkn}{2}} e^{-\gamma}d\frac{\mathrm{d}\gamma}{\gamma} 
\propto 
\left(\frac{2}{\operatorname{tr(\Sigma^{-1}x)}} \right)^{\frac{dkn}{2}}.   
\end{equation}
\vspace{0.2cm}

Note that the formula (\ref{eq:denDimD}) is valid in for any dimension $d$. We have $\operatorname{tr}(\Sigma^{-1} x)= \det(\Sigma)^{-\frac{1}{d}} \operatorname{tr}(\bar\Sigma^{-\tfrac12}x\bar\Sigma^{-\tfrac12})$ and 
since the matrix $\bar\Sigma^{-\tfrac12}x\bar\Sigma^{-\tfrac12}$ is in $S$, there exists  $R\in H$ such that 
$$
R\bar\Sigma^{-\tfrac12}x\bar\Sigma^{-\tfrac12}R^*=\begin{pmatrix}
    \lambda&0\\0&\frac{1}{\lambda}
\end{pmatrix}.
$$

We have then $\operatorname{tr}(\bar \Sigma^{-1}x)=\lambda + \frac{1}{\lambda}$. By Eq.(\ref{eq:dist}) in section 3,
$$
d(x,\bar\Sigma)\propto \|\log(\bar\Sigma^{-\tfrac12}x\bar\Sigma^{-\tfrac12})\|=\|\log(R\bar\Sigma^{-\tfrac12}x\bar\Sigma^{-\tfrac12}R^*)\|=\sqrt 2|\log\lambda|.
$$

Since we can choose the multiplicative factor in Eq.(\ref{eq:dist}), suppose that $d(x,\bar\Sigma)=|\log\lambda|$. We have then
$$\left(\frac{2}{\operatorname{tr(\Sigma^{-1}x)}} \right)^{\frac{dkn}{2}} =  \left(\frac{2(\det \Sigma)^{\frac{1}{2}}}{\lambda + \frac{1}{\lambda}} \right)^{kn} \propto \left( \frac{e^{d(x,\bar \Sigma)}+e^{-d(x,\bar\Sigma)}}{2}\right)^{-kn},$$

which leads to the following theorem.

\begin{theorem}
The distance $d(.,.)$ can be normalized such that for all $x\in \mathcal{S}$, $\Sigma\in \mathcal{P}_2(K)$ and $n\in \mathbb{N}, n>2$,

    $$f_{P\mathcal{W}}(x;\Sigma,n) \propto \cosh(d(x,\bar\Sigma))^{-kn}.$$
\end{theorem}

As announced, the function $h_{\Sigma,n}=\cosh^{-kn}$ does not depends on $\Sigma$.

\section{Conclusion}

We exhibited simple links between projective Wishart distributions and the affine invariant distance of positive definite matrices of constant determinant. Our future researches will focus on two aspects. Firstly, we will investigate the convergence of the estimation of the parameters of projective Wishart distributions. Secondly, we will investigate the use of the geometric properties of these distributions in signal processing applications.

%
%
%

\begin{thebibliography}{8}

\bibitem{Kollo}
Kollo, T., Dietrich Rosen, D.: Advanced multivariate statistics with matrices. Springer (2005).

\bibitem{Muirhead}
Muirhead, R. J.: Aspects of Multivariate Statistical Theory. Wiley Series in Probability and Statistics (1982).

\bibitem{Bhatia}
Bhatia, R.: Positive Definite Matrices. Princeton University Press, Princeton (2007). \doi{10.1515/9781400827787}

\bibitem{Terras}
Terras, A.: Harmonic Analysis on Symmetric Spaces and Applications II. 2nd edn. Springer New York, NY, (1988)

\bibitem{Faraut}
Faraut, J.: Analysis on symmetric cones. Oxford mathematical monographs (1994).

\bibitem{Zhuang}
Zhuang, L., Walden, A. T.: Sample Mean Versus Sample Fréchet Mean for Combining Complex Wishart Matrices: A Statistical Study. In IEEE Transactions on Signal Processing, \textbf{65}(17), 4551--4561 (2017). \doi{10.1109/TSP.2017.2713763}.

\bibitem{Wishart}
Wishart, J.: The Generalised Product Moment Distribution in Samples from a Normal, Multivariate Population. Biometrika, \textbf{20}, 32--52 (1928). 

\bibitem{Pennec}
Pennec, X., Fillard, P., Ayache, N.: A Riemannian Framework for Tensor Computing. Int J Comput Vision \textbf{66}, 41–-66 (2006). \doi{10.1007/s11263-005-3222-z}

\bibitem{Thanwerdas}
Thanwerdas, Y., Pennec, X.: Is affine-invariance well defined on SPD matrices? A principled continuum of metrics. Geometric Science of Information: 4th International Conference, GSI 2019, Toulouse, France, Proceedings 4. Springer International Publishing (2019).

\bibitem{Afsari}
Afsari, B. Riemannian $L^p$ center of mass: existence, uniqueness, and convexity. Proceedings of the American Mathematical Society, 2011, \textbf{139}(2), 655--673.

\bibitem{Arnaudon}
Arnaudon, M., Barbaresco, F. and Yang, L.: Riemannian Medians and Means With Applications to Radar Signal Processing. In IEEE Journal of Selected Topics in Signal Processing, \textbf{7}(4), 595--604 (2013), \doi{10.1109/JSTSP.2013.2261798}

\bibitem{James}
James, A.T.: Distributions of Matrix Variates and Latent Roots Derived from Normal Samples. Ann. Math. Statist. \textbf{35}(2), 475 -- 501 (1964). \doi{10.1214/aoms/1177703550}

\end{thebibliography}
%

\end{document}